\documentclass[12pt]{article}

\usepackage{microtype}
\usepackage{amsfonts}
\usepackage{graphicx}
\usepackage{enumitem}
\usepackage{amsmath}
\usepackage{amsthm}
\usepackage{amssymb}
\usepackage{hyperref}

\theoremstyle{plain}
\newtheorem{thm}{Theorem}
\theoremstyle{remark}
\newtheorem{rem}{Remark}

\newcommand{\ml}{\mathcal}
\newcommand{\mb}{\mathbb}
\newcommand{\bs}{\boldsymbol}

\begin{document}

\title{Avoiding unwanted results in locally 
linear embedding: A new understanding of regularization}
\author{Liren Lin\\ \\
Department of Applied Mathematics,\\
National Sun Yat-sen University, Taiwan\\
\tt{lirenlin2017@gmail.com}}
\date{}
\maketitle

\begin{abstract}
We demonstrate that locally linear embedding (LLE) 
inherently admits some unwanted results when no regularization 
is used, even for cases in which regularization is not 
supposed to be needed in the original algorithm. The existence 
of one special type of result, which we call ``projection 
pattern'', is mathematically proved in the situation that 
an exact local linear relation is achieved in each neighborhood of 
the data. These special patterns as well as some other bizarre 
results that may occur in more general situations
are shown by numerical examples on the Swiss roll with a hole 
embedded in a high dimensional space. It is observed that 
all these bad results can be effectively prevented 
by using regularization.
\end{abstract}

\section{Introduction}

Let $\ml{X}=\{x_i\}_{i=1}^N$ be a collection of points 
in some high dimensional space $\mb{R}^D$. The general goal of 
manifold learning (or nonlinear dimensionality reduction) is to 
find for $\ml{X}$ a representation $\ml{Y}=\{y_i\}_{i=1}^N$ in 
some lower dimensional $\mb{R}^d$, under the assumption 
that $\ml{X}$ lies on some unknown submanifold of 
$\mb{R}^D$. 

Locally linear embedding (LLE) \cite{lle}, due to its 
simplicity in idea as well as efficiency in implementation, is a 
popular manifold learning method, which has been 
studied a lot and has many variant and modified 
versions \cite{hlle2003,de2003supervised,kouro2005,
chang2006robust, mlle2007,chenliu2011,wu2018,wang2015bearing}. 
The goal of this paper is to point out a fundamental problem,
which seems to have not been addressed yet, 
and then propose a simple solution.
Precisely, we will demonstrate 
that LLE inherently admits some unwanted results when no 
regularization is used, even for cases in which regularization 
is not supposed to be needed in the original algorithm.
And the solution is just to use regularization in any case.

The existence of one special type of unwanted results, 
which we call ``projection patterns'', will be mathematically 
proved in the situation that an exact local linear relation 
is achieved in each neighborhood of the dataset $\ml X$. 
Such projection patterns, as indicated by their name, 
are basically direct projections of $\ml{X}$ from $\mb{R}^D$ 
onto a $d$-dimensional hyperplane, ignoring the  
geometry of the data totally.
Due to the using of regularization, they do not appear 
for most artificial datasets, which reside in some low 
dimensional $\mb{R}^D$ (mostly $\mb{R}^3$).
The problem is that for general data in a high 
dimensional $\mb{R}^D$, the 
regularization is supposed to be unnecessary and not employed. 
We will show by numerical examples that this practice is risky.
The idea is simple:
performing some embeddings of 
the Swiss roll with a hole into a high dimensional $\mb{R}^D$ 
such that 
regularization is not used in the original LLE algorithm, and 
then comparing the results to those with regularization
being used. It turns out that, if regularization is not used, 
the projection phenomenon always occurs as predicted when an isometric embedding is applied, while if some further perturbation is added 
to the embedding, more bizarre results may appear.
By contrast, with regularization, all the bad results are
effectively prevented.

\section{Projection patterns in LLE}\label{sec:lle}

Let us begin with a quick review of the LLE procedure. As in the introduction, 
the dataset $\ml X=\{x_{i}\}_{i=1}^{N}$ is a subset of 
$\mb R^{D}$ and we want to find for it a representation 
$\ml Y=\{y_{i}\}_{i=1}^{N}$ in some lower dimensional 
$\mb R^{d}$. 
\begin{enumerate}[align=left]
\item[Step 1.] For each $i=1,\ldots,N$, let 
$\ml U_{i}=\{x_{i_{1}},\ldots,x_{i_{k}}\}
\subset\ml X\setminus\{x_{i}\}$ be a $k$-nearest 
neighborhood of $x_{i}$. (For simplicity we consider
$k$ to be fixed for all $i$.)
\item[Step 2.] Set 
$w^{(i)}=(w_1^{(i)},\ldots,w_k^{(i)})\in\mb{R}^k$ to be 
a solution of the problem 
\begin{equation}
\underset{(w_{1},\ldots,w_{k})\in\mb R^{k}}{\mbox{argmin}}
\|x_{i}-\sum_{j=1}^{k}w_{j}x_{i_{j}}\|^{2}\quad\mbox{s.t.}
\quad\sum_{j=1}^{k}w_{j}=1.\tag{P1}\label{eq:p1}
\end{equation}
\item[Step 3.] With $w^{(i)}$ given from Step 2, set $\ml Y$ to be
a solution of the problem
\begin{equation}
\underset{\{y_{i}\}_{i=1}^{N}\subset\mb R^{d}}{\mbox{argmin}}
\sum_{i=1}^{N}\|y_{i}-\sum_{j=1}^{k}w_{j}^{(i)}y_{i_{j}}\|^{2}
\quad\mbox{s.t.}\quad YY^{T}=I,\tag{P2}\label{eq:p2}
\end{equation}
where $Y$ denotes the matrix 
$\left[y_{1}\ \cdots\ y_{N}\right]\in\mb R^{d\times N}$,
and $I$ is the identity matrix.
\end{enumerate}
For convenience, in the following we reserve the index $i$ 
to be the one which runs from $1$ to $N$, and $i_{1},\ldots,i_{k}$
are the indices for the corresponding $k$-nearest points being 
chosen in Step 1. 

The following remark suggests a possible problem of the LLE 
algorithm, which turns out to be what happens to the 
``unwanted results'' that will be discussed later.

\begin{rem}\label{rem:tei} 
Step 3 is stated in accordance with the original
idea. It is not what is exactly executed in the algorithm.
To be precise, write
\begin{align*}
\sum_{i=1}^{N}\|y_{i}-\sum_{j=1}^{k}w_{j}^{(i)}y_{i_{j}}\|^{2}
=\|(I-W)Y^T\|^2,
\end{align*}
where $W=(W_{ij})$ is the $N\times N$ matrix defined by 
\begin{align*}
W_{is} = \bigg\{
\begin{array}{cl}
w_j^{(i)} &  \mbox{if }\ s= i_j\ (j=1,\ldots,k)\\[3pt]
0 &  \mbox{if }\ s\notin \{i_{1},\ldots,i_{k}\},
\end{array}
\end{align*}
and $\|\cdot\|$ for matrices denotes the Frobenius norm.
Then, it is known that solutions of \eqref{eq:p2} are given by 
\begin{equation}
Y^{T}=\left[g_{1}\ \cdots\ g_{d}\right],\label{eq:lley}
\end{equation}
where $g_{1},\ldots,g_{d}$ are any selection of 
orthonormal eigenvectors of 
$(I-W)^T(I-W)$ with corresponding eigenvalues 
$0=\lambda_1\le \lambda_2\le \ldots\le\lambda_d$.
However, in this way one of the $g_{j}$'s 
might be $\frac{1}{\sqrt{N}}\bs 1_{N}$, corresponding to the 
eigenvalue $0$. Here $\bs{1}_N\in \mb R^N$ denotes the vector 
whose components all equal $1$.
This eigenvector is redundant for our 
purpose, since the dimension reduced data $\ml Y$ given by 
\eqref{eq:lley} is trivially isometric to that given by 
$[g_2\ g_3\ \cdots\ g_d]$. Therefore, to find
an effective $d$-dimensional representation for $\ml X$, we 
would like to exclude $\frac{1}{\sqrt{N}}\bs 1_N$ from our 
selection of eigenvectors. In the LLE algorithm, what is
executed is to find out the first $d+1$ eigenvectors, 
and set $g_1,\ldots,g_d$ to be the second to the $(d+1)$-th ones 
respectively, in the anticipation that the first eigenvector 
should be $\frac{1}{\sqrt{N}}\bs{1}_N$. However, note that this 
approach may break down when the multiplicity of the 
eigenvalue $0$ is greater than one, hence there
are more than one eigenvectors corresponding to it. 
In this situation, it is 
not guaranteed that the first eigenvector being found by computer
is $\frac{1}{\sqrt{N}}\bs 1_{N}$. 
Indeed, it is not necessary that 
any one of the eigenvectors been found must be 
$\frac{1}{\sqrt{N}}\bs 1_{N}$. 
\end{rem}

Now we turn to the resolution of \eqref{eq:p1}. It can be 
rewritten as 
\[
\underset{w\in\mb R^{k}}{\mbox{argmin}}\|Z_{i}w\|^{2}\quad
\mbox{s.t.}\quad\bs 1_{k}^{T}w=1,
\]
where
\[
Z_{i}=\begin{bmatrix}x_{i_{1}}-x_{i} & \cdots & x_{i_{k}}-x_{i}\end{bmatrix}\in\mb R^{D\times k}.
\]
Let $C_{i}=Z_{i}^{T}Z_{i}$, then 
$\|Z_i w\|^2 = w^T Z_i^T Z_i w = w^T C_i w$.
By the method of Lagrange multiplier, any minimizer $w$ of \eqref{eq:p1}
satisfies
\begin{equation}\label{eq:lm}
\left\{ \begin{aligned} & C_i w=\lambda\bs 1_{k}\\
 & \bs 1_{k}^{T}w=1
\end{aligned}
\right.
\end{equation}
for some $\lambda\in\mb R$. If $C_{i}$
is invertible, \eqref{eq:lm} has a unique solution 
\[
w^{(i)}=\frac{C_{i}^{-1}\bs 1_{k}}{\bs 1_{k}^{T}C_{i}^{-1}\bs 1_{k}}.
\]
On the other hand, if $C_{i}$ is singular, \eqref{eq:lm} may have
multiple solutions. In this situation it is suggested that we 
add a small regularization term $\epsilon I$ to $C_{i}$, and set 
\begin{align}\label{eq:rew}
w^{(i)}=w^{(i)}(\epsilon)=\frac{(C_{i}+\epsilon I)^{-1}\bs 1_{k}}
{\bs 1_{k}^{T}(C_{i}+\epsilon I)^{-1}\bs 1_{k}}.
\end{align}

For efficiency, in the LLE algorithm whether the regularization 
term is used or not is not precisely determined by the 
invertibility or not of $C_i$, but by the following rule: 
used if $k>D$, and not used if $D\ge k$. Note that this rule is only 
a convenient one but not in exact accordance with the original
idea. Precisely, if $k>D$, then it's true that $C_i$ must be singular, 
but $D\ge k$ does not guarantee that $C_i$ is non-singular. 
Nevertheless, as we will advocate using regularization 
(namely the formula \eqref{eq:rew})
no matter $C_i$ is singular or not, we do not care about 
this problem.

A common perception of \eqref{eq:rew} is that it provides
a stable way to solve \eqref{eq:lm}, at the cost of a small amount of error.
In fact, 
\[
w^{(i)}(0^{+}):=\lim_{\epsilon\to0^{+}}w^{(i)}(\epsilon)
\]
exists and is an exact solution of \eqref{eq:lm}. 
In this thread, it seems natural to 
expect better performance of LLE if $w^{(i)}(\epsilon)$ is
replaced by $w^{(i)}(0^{+})$. However, that is not the case.
In Figure \ref{fig:1},
we give some numerical examples 
on the Swiss roll with a hole for 
different values of $\epsilon$, and we see that 
when $\epsilon$ is too small, the resulting $\ml Y$ becomes 
far from what it ought to be. Actually, when $\epsilon$ approaches zero, 
we see that $\ml Y$ converges to a ``projection pattern''. 
That is, it looks like $\ml Y$ is simply some 
projection of the Swiss roll onto the plane, ignoring the 
rolled-up nature of the original data.
We are going to explain that this phenomenon is not a 
coincidence, and is not due to any instability
problem in solving \eqref{eq:lm}. In fact, such projection 
mappings are inherently allowed in the LLE procedure 
when no regularization is used.
\begin{figure}
\includegraphics[width=1\textwidth]{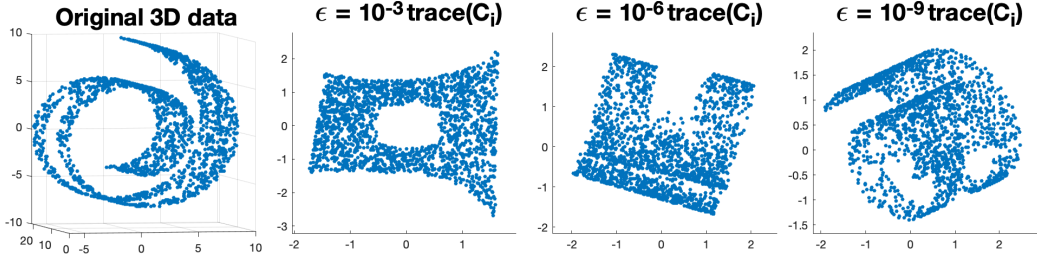}
\caption{Results of applying LLE to the Swiss roll with a hole 
with different sizes of $\epsilon$. Notice that the rightmost 
image is a 2D one.
\label{fig:1}}
\end{figure}

To explain the phenomenon, we formulate precisely two assumptions:
\begin{enumerate}[label=(A\arabic{enumi}), ref = (A\arabic{enumi})]
\item $x_{i}-\sum_{j=1}^{k}w_{j}^{(i)}x_{i_{j}}=0$ for all $i=1,\ldots,N$. \label{a1}
\item $\mbox{rank}(X)\ge d$, where 
$X:=\begin{bmatrix}x_{1} & \cdots & x_{N}\end{bmatrix}\in\mb R^{D\times N}$. \label{a2}
\end{enumerate}
Assumption \ref{a1} is the same as saying that zero is achieved as the minimum in  
\eqref{eq:p1} for all $i$. Note that this is the normal situation for cases with 
$k>D$ -- the same cases for which regularization is employed 
in the LLE algorithm. As for \ref{a2}, it 
is a natural assumption which holds in all cases of interest. 
To see it, suppose $\mbox{rank}(X)=r<d$, then some $r$ columns from $X$
span all the others, and hence $\ml X$ lies entirely on some $r$
dimensional subspace. In this case there is no point to pursue an embedding
of $\ml X$ in $\mb R^{d}$.

Recall that 
$\mbox{rank}(XX^{T})=\mbox{rank}(X)$,
and hence \ref{a2} implies $XX^{T}$ has at least $d$ nonzero
eigenvalues (counting multiplicity). Moreover, since $XX^{T}$ is 
positive semi-definite, all of its nonzero eigenvalues are positive. 
With these in mind, we can now state our key observation.
\begin{thm}
\label{thm:m1} Assume \ref{a1} and \ref{a2}. 
Let $u_{1},\ldots,u_{d}\in\mb R^{D}$ be any collection of 
orthonormal eigenvectors of $XX^{T}$ whose corresponding
eigenvalues, denoted $\lambda_{1},\ldots,\lambda_{d}$, are all positive. 
Then $\ml Y:=\{Ax_{i}\}_{i=1}^{N}$
is a solution of \eqref{eq:p2}, where $A\in\mb{R}^{d\times D}$
is given by
\begin{equation}
A=\begin{bmatrix}\lambda_{1}^{-1/2}\\
 & \ddots\\
 &  & \lambda_{d}^{-1/2}
\end{bmatrix}\begin{bmatrix}u_{1}^{T}\\
\vdots\\
u_{d}^{T}
\end{bmatrix}.\label{eq:ad}
\end{equation}
\end{thm}

\begin{proof}
The local linear relations in \ref{a1} 
are preserved by applying $A$, that is 
\[
Ax_{i}-\sum_{j=1}^{k}w_{j}^{(i)}Ax_{i_{j}}=0\quad\forall\,i.
\]
As a consequence, by setting $\ml Y=\{y_{i}\}_{i=1}^{N}=\{Ax_{i}\}_{i=1}^{N}$,
we have
\[
\sum_{i=1}^{N}\|y_{i}-\sum_{j=1}^{k}w_{j}^{(i)}y_{i_{j}}\|^{2}=0,
\]
and hence $\{y_{i}\}_{i=1}^{N}$ is a minimizer of the cost function in 
Problem \eqref{eq:p2}. It remains to show that the constraint 
$YY^{T}=I$ is also satisfied. Note that $Y=AX$, 
and hence the constraint is
\[
AXX^{T}A^{T}=I.
\]
The validity of this equality is easy to check from the definition
\eqref{eq:ad} of $A$, and the fact that 
$XX^{T}u_{\ell}=\lambda_{\ell}u_{\ell}$
for $\ell=1,\ldots,d$.
\end{proof}
Now, what does $\ml Y = \{Ax_{i}\}_{i=1}^{N}$ in Theorem 
\ref{thm:m1} look like? Note that
\begin{equation}
x\mapsto\sum_{\ell=1}^{d}(x\cdot u_{\ell})u_{\ell}\label{eq:op}
\end{equation}
is the orthogonal projection of $x\in\mb R^{D}$ onto the 
$d$-dimensional subspace $\mbox{span}(u_{1},\ldots,u_{d})$. 
By endowing this subspace with its own coordinate system with respect to
$u_{1},\ldots,u_{d}$, we can regard the projection as a mapping from
$\mb R^{D}$ into $\mb R^{d}$, which is expressed by the coefficients
in \eqref{eq:op}: 
\[
x\mapsto\begin{bmatrix}x\cdot u_{1}\\
\vdots\\
x\cdot u_{d}
\end{bmatrix}=\begin{bmatrix}u_{1}^{T}\\
\vdots\\
u_{d}^{T}
\end{bmatrix}x.
\]
This is the first half of the action of $A$. The second half 
is a further multiplication of the diagonal matrix 
$\mbox{diag}(\lambda_{1}^{-1/2},\ldots,\lambda_{d}^{-1/2})$,
which is nothing but performing some rescalings of the 
coordinates. In summary, $x\mapsto Ax$ is an 
orthogonal projection from $\mb R^{D}$ to $\mb R^{d}$ 
followed by some rescalings of coordinates in $\mb R^{d}$.
This explains the projection phenomenon observed. 

\begin{rem}\label{rem:cc}
What Theorem \ref{thm:m1} tells us is that, when 
\ref{a1} and \ref{a2} hold, some projection patterns are
solutions of \eqref{eq:p1}, and hence are candidates for
being the result of LLE. Logically speaking, it does not preclude 
the possiblity that there are other kinds of results (which we do 
not know). 
\end{rem}

\section{Numerical examples for high dimensional data}\label{sec:num}

The projection phenomenon does not 
pertain exclusively to cases with $k>D$. 
For one thing, even if $D\ge k$, if $\ml{X}$ lies on some 
$D'$-dimensional hyperplane with $D'<k$, 
then \ref{a1} still holds in general (we take 
\ref{a2} for granted and will no longer mention it), and 
Theorem \ref{thm:m1} applies.
Of course, such a degenerate case is of little interest. 
But for another, it is conceivable that 
the projection effect may still have its influence
when \ref{a1} is only approximately true while to a high degree.
This statement can partly be supported by the third and fourth images
in Figure \ref{fig:1}, where $\epsilon$
is very small and \ref{a1} is almost true.
However, those images correspond to the situation $k>D=3$. 

To acquire some direct evidences about what could 
happen to high dimensional data, 
we have performed several experiments in which 
the Swiss roll with a hole was first embedded in a high dimensional 
$\mb{R}^D$, and then applied LLE to it with $k$ chosen 
to be smaller than $D$.
For our purpose, we only considered isometric embeddings 
and perturbed isometric embeddings
so that the embedded datasets are basically the same 
Swiss roll with a hole.
It turns out that for isometrically embedded data 
(which lie entirely on some three dimensional 
hyperplane in $\mb R^D$), the projection phenomenon is obvious 
(see however Remark \ref{c1} at the end of this section).
Somewhat surprising to us is that when rather small 
perturbations are added to the isometric embeddings, the outcomes 
become quite unpredictable but not merely perturbations of 
projection patterns. 
It is observed that all these unwanted results -- projection
patterns and others -- are associated with the problem mentioned
in Remark \ref{rem:tei}. That is, the matrix $(I-W)^T(I-W)$
corresponding to them is very singular.
And in any case, by using regularization we see this problem can be 
effectively avoided, and considerably improved results can be obtained.
See Figure \ref{fig:2} for some examples, the details of which  
are given in the next paragraph. 
\begin{figure}
\includegraphics[width=\textwidth]{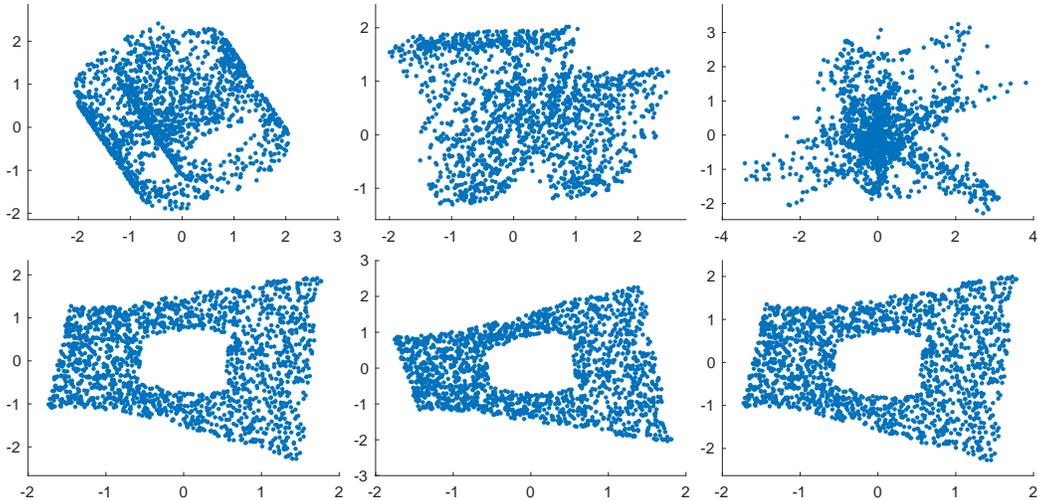}
\caption{Examples of applying LLE to the Swiss roll with a hole
embedded in dimensional spaces. From left to right, 
embeddings $E_1,E_2$ and $E_3$ are applied respectively (see descriptions in the text). 
Top row: results of the original LLE; Bottom row: corresponding results with regularization.
\label{fig:2}}
\end{figure}

Let $\ml X =\{x_i\}$ represents a Swiss roll with a hole dataset 
in $\mb{R}^3$. 
Figure \ref{fig:2} shows, from left to right, the results of 
applying LLE to 
three different datasets $E_1\ml X,E_2\ml X$ and $E_3\ml X$ 
with $k=12$.
The top row are results of the original LLE (no regularization), 
and the bottom
row are corresponding results for which the regularized weight 
vector \eqref{eq:rew}  is used for each $i$, with 
$\epsilon = 10^{-3}\mbox{trace}(C_i)$.
$E_i$, $i=1,2,3$, are given as follows:
\begin{itemize}
\item $E_1$ is some randomly generated $18\times 3$ matrix whose 
columns form an orthonormal set of vectors. 
This gives rise to a linear isometric mapping from $\mb R^3$ into 
$\mb R^{18}$.
\item $E_2$ is the mapping from $\mb R^3$ to $\mb R^{19}$ obtained 
by perturbing $E_1$ in an extra dimension:
\begin{align*}
E_2 x = (E_1 x, 0.1\sin (\sum_{j=1}^{18}(E_1 x)_j)),
\end{align*}
where $(E_1 x)_j$ denotes the $j$-th component of $E_1 x$.
\item  $E_3:\mb R^3\to\mb R^{18}$ is another perturbation of $E_1$ given by
\begin{align*}
E_3 x = E_1 x + 0.1(\sin((E_1 x)_1),\sin((E_1x)_2),\ldots,\sin((E_1x)_{18})).
\end{align*}
\end{itemize}
Here are some remarks:
\begin{enumerate}[label=(\alph{enumi}), ref = (\alph{enumi})]
\item For isometrically embedded data, the results without
regularization are not always typical projection patterns 
as the left top image in Figure \ref{fig:2}. Sometimes additional 
deformation or distortion is also present (see Figure \ref{fig:3}). 
We do not know if it reflects the fact that there are 
other possible results than 
projection patterns (cf. Remark \ref{rem:cc}), or is simply 
the cause of errors in numerical simulations.\label{c1}
\begin{figure}
\includegraphics[width=\textwidth]{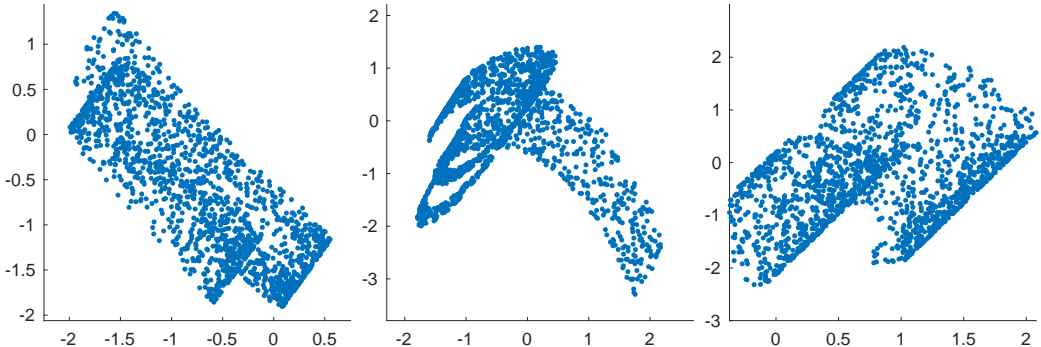}
\caption{Some ``non-typical'' projection patterns 
obtained by applying LLE to the Swiss roll with a hole 
isometrically embedded in a high dimensional space. 
\label{fig:3}}
\end{figure}
\item The using of sine function in $E_2$ and $E_3$ is a random choice
and bears no significance. We can still see the projection effect 
for $E_2\ml X$ without regularization (top middle image of Figure \ref{fig:2}), 
while the corresponding result for $E_3\ml X$ (top right) is totally inexplicable. 
One difference between $E_2\ml X$ and $E_3\ml X$ that might
be important is that the former lies on 
a 4-dimensional hyperplane in $\mb R^{19}$, while the latter does not 
lie on any lower dimensional hyperplane in $\mb R^{18}$. 
\item The above examples show clearly that the condition $k<D$ doesn't imply 
regularization is unnecessary.
However, as is mentioned, the criterion based on the size relation between 
$k$ and $D$
is only something for convenience. How about using the original idea:
performing regularization or not according to whether $C_i$ is singular
or not? Experiments (not shown here) show that 
it doesn't help either. 
In fact, for $E_3\ml X$, the majority of the $C_i$'s
are already non-singular, and even if all the singular
$C_i$'s are regularized (with non-singular $C_i$'s unchanged), 
the result still looks as bizarre as the original one.
\end{enumerate}

\section{Conclusion}\label{sec:con}

We have demonstrated that LLE inherently admits some unwanted results
if no regularization is used, even 
for cases in which regularization is supposed to be 
unnecessary in the original algorithm. 
The true merit of regularization is hence not (or at least not merely)
for solving \eqref{eq:lm} stably, at the cost of a small
amount of error. 
On the contrary, by deliberately distorting the local linear relations,
it protects LLE from some bad results. As a consequence, 
we suggest that one uses regularization in any case when applying
LLE. 
Of course, our investigation is far from comprehensive. 
More examples, especially high dimensional real world data, 
should be examined. Moreover, using regularization 
alone is in no way a promise of good results. 

\section*{Acknowledgment}
This work is supported by Ministry of Science 
and Technology of Taiwan under grant number MOST110-2636-M-110-005-. 
The author would also like to thank Chih-Wei Chen
for valuable discussions.

\bibliographystyle{plain}
\bibliography{lle}

\end{document}